\documentclass[letterpaper, 10 pt, conference]{ieeeconf}
\IEEEoverridecommandlockouts
\let\labelindent\relax
\usepackage{enumitem}
\usepackage{balance}
\usepackage[font={small}]{caption}
\usepackage{subcaption}
\usepackage{array}
\usepackage{textcomp}
\usepackage{mathtools, nccmath}
\usepackage{graphicx}
\usepackage{amsfonts}
\usepackage{amsmath}
\usepackage{amssymb}
\usepackage{algorithm}
\usepackage{algorithmic}
\usepackage{hyperref}
\usepackage{tikz}
\usepackage{arydshln}
\usepackage{multirow}
\usepackage{bm}
\usepackage{epstopdf}
\usepackage{cite}
\usepackage{siunitx}

\usepackage{amsthm}

\newtheorem{theorem}{Theorem}

\newtheorem{remark}{Remark}
\theoremstyle{definition}
\newtheorem{definition}{Definition}
\newtheorem{problem}{Problem}

\title{\LARGE \bf
Auxiliary-Variable Adaptive Control Barrier Functions \\for Safety Critical Systems}

\author{Shuo Liu$^{1}$, Wei Xiao$^{2}$ and Calin A. Belta$^{1}$% <-this % stops a space
\thanks{This work was supported in part by the NSF under grant IIS-2024606.}
\thanks{$^{1}$S. Liu and C. Belta are with the department of Mechanical Engineering, Boston
University, Brookline, MA, 02215, USA. 
        {\tt\small \{liushuo, cbelta\}@bu.edu}}%
\thanks{$^{2}$W. Xiao is with the Computer Science and Artificial Intelligence Lab, Massachusetts Institute of Technology. 
        {\tt\small weixy@mit.edu}}%
}

\begin{document} 
\maketitle
%{\color{blue} calin' color}
%%%%%%%%%%%%%%%%%%%%%%%%%%%%%%%%%%%%%%%%%%%%%%%%%%%%%%%%%%%%%%%%%%%%%%%%%%%%%%%%
\begin{abstract}
This paper studies safety guarantees for systems with time-varying control bounds. It has been shown that optimizing quadratic costs subject to state and control constraints can be reduced to a sequence of Quadratic Programs (QPs) using Control Barrier Functions (CBFs). One of the main challenges in this method is that the CBF-based QP could easily become infeasible under tight control bounds, especially when the control bounds are time-varying. The recently proposed adaptive CBFs have addressed such infeasibility issues, but require extensive and non-trivial hyperparameter tuning for the CBF-based QP and may introduce overshooting control near the boundaries of safe sets. To address these issues, we propose a new type of adaptive CBFs called Auxiliary-Variable Adaptive CBFs (AVCBFs). Specifically, we introduce an auxiliary variable that multiplies each CBF itself, and define dynamics for the auxiliary variable to adapt it in constructing the corresponding CBF constraint. In this way, we can improve the feasibility of the CBF-based QP while avoiding extensive parameter tuning with non-overshooting control since the formulation is identical to classical CBF methods.
 We demonstrate the advantages of using AVCBFs and compare them with existing techniques on an Adaptive Cruise Control (ACC) problem with time-varying control bounds. 
% Safety is one of the fundamental challenges in control theory. Recently, the design of safety-critical optimal controller for affine control systems was realized by formulating Control Barrier Functions (CBFs) into Quadratic programs (QPs), which always leads to infeasible condition under conservative input constraints.
% Existing approaches usually focus on the feasibility or the safety for the optimization problem, therefore guaranteeing the safety of control strategy and the feasibility of the QPs under conservative input constraints is still challenging to current approaches, not to mention the complexity of parameter-tuning process.
% In this paper, we propose Auxiliary Variable Based CBFs (AVCBFs) which utilize auxiliary variables for the system dynamics and for High Order Control Barrier Functions (HOCBFs) to enhance the feasibility of the QPs while to guarantee the safety. We demonstrate the advantages of using AVCBFs over one of the state of the arts by applying both approaches to an Adaptive Cruise Control (ACC) problem with conservative and time-varying input constraints. The numerical results validate that our proposed method can generate a safe, smooth and adaptive controller without defining excessive additional constraints, which enormously simplifies the parameter-tuning process.
\end{abstract}
\section{Introduction}
\label{sec:Introduction}

Barrier functions (BFs) are Lyapunov-like functions \cite{tee2009barrier} whose use can be traced back to optimization problems \cite{boyd2004convex}. They have been utilized to prove set invariance \cite{aubin2011viability} \cite{prajna2007framework} and to derive multi-objective control \cite{panagou2013multi} \cite{wang2016multi}. More recently, a less restrictive form of a barrier function, which is allowed to grow when far away from the boundary of the set, was proposed in \cite{ames2016control}. Another approach that allows
a barrier function to reach the boundary of
the set was proposed in \cite{glotfelter2017nonsmooth}. 

In \cite{boyd2004convex}, Control Barrier Functions (CBFs) are extensions of BFs used to render a set forward invariant for an affine control system. It has been shown that stabilizing a control-affine system to admissible states while optimizing a quadratic cost subject to state and control constraints can be reduced to a sequence of Quadratic Programs (QPs) \cite{ames2016control} by unifying CBFs and Control Lyapunov Functions (CLFs) \cite{ames2012control}. 
Exponential CBFs \cite{nguyen2016exponential} are introduced in order to adapt CBFs to high relative degree systems. A more general form of exponential CBFs, called High Order CBFs (HOCBFs), has been proposed in \cite{xiao2021high}. The CBF method has been widely used to enforce safety in many applications, including adaptive cruise control \cite{ames2016control}, bipedal robot walking, \cite{hsu2015control} and robot swarming \cite{borrmann2015control}. However, the aforementioned CBF-based QP might be infeasible in the presence of tight or time-varying control bounds due to the conflict between CBF constraints and control bounds.

There are several approaches that aim to enhance the feasibility of the CBF method while guaranteeing safety. One can formulate CBFs as constraints under a Nonlinear Model Predictive Control (NMPC) framework, which allows the controller to predict future state information up to a horizon larger than one. This leads to a less aggressive control strategy \cite{zeng2021enhancing}. However, the corresponding
optimization is overall nonlinear and non-convex, 
which could be computationally expensive for
nonlinear systems. A convex MPC with linearized, discrete-time CBFs under an iterative approach was proposed in \cite{liu2023iterative} to address the above challenges, but this comes at the price of losing safety guarantees. The works in  \cite{gurriet2018online,singletary2019online,gurriet2020scalable,chen2021backup} recently developed approaches
in which a known backup set or backup policy is defined that can be used to extend the safe set to a larger viable set
to enhance the feasible space of the system in a finite time horizon
under input constraints. This backup approach has further been generalized to infinite time horizons \cite{squires2018constructive} \cite{breeden2021high}. One limitation of these approaches is that they require prior knowledge of finding appropriate backup sets, policy or nominal control law, which are difficult to be predefined. Another limitation is that they only focus on feasibility, which may introduce over-aggressive or over-conservative control strategies. Sufficient conditions have been proposed in \cite{xiao2022sufficient} to guarantee the feasibility of the CBF-based QP, but they may be hard to find for general constrained control problems. All these approaches only consider time-invariant control limitations.

In order to account for time-varying control bounds, adaptive CBFs (aCBFs) \cite{xiao2021adaptive} have been proposed by
introducing penalty functions in HOCBFs constraints, which provide flexible and adaptive control strategies over time. However, this approach requires extensive parameter tuning. 
% Unlike the AVBCBFs {\color{blue} acronym not defined. also we do not need B for "based", it should be AVCBF} proposed in this paper, but additional constraint design and complicated parameter tuning process However, choosing appropriate penalty variables for class $\kappa$ functions may not be straightforward and does need additional constraint design and complicated parameter tuning process.
To address this issue,
we propose a novel type of aCBFs to safety-critical control problems. Specifically, the contributions of this paper are as follows:

\begin{itemize}
\item We propose Auxiliary-Variable Adaptive CBFs (AVCBFs) that can improve the feasibility of the CBF method under tight and time-varying control bounds.

\item We show that the proposed AVCBFs are analogous to existing CBF methods such that excessive additional constraints are not required. Most importantly, the AVCBFs preserve the adaptive property of aCBFs \cite{xiao2021adaptive}, while generating non-overshooting control policies near the boundaries of safe sets.

\item We demonstrate the effectiveness of the proposed AVCBFs on an adaptive cruise control problem with tight and time-varying control bounds, and compare it with existing CBF methods. The results show that the proposed approach can generate smoother and more adaptive control compared to existing methods, without requiring design of excessive additional constraints and complicated parameter-tuning procedures.
\end{itemize}

\section{Preliminaries}
\label{sec:preliminaries}

Consider an affine control system expressed as 
\begin{equation}
\label{eq:affine-control-system}
\dot{\boldsymbol{x}}=f(\boldsymbol{x})+g(\boldsymbol{x})\boldsymbol{u},
\end{equation}
 where $\boldsymbol{x}\in \mathbb{R}^{n}, f:\mathbb{R}^{n}\to\mathbb{R}^{n}$ and $g:\mathbb{R}^{n}\to\mathbb{R}^{n\times q}$ are locally Lipschitz, and $\boldsymbol{u}\in \mathcal U\subset \mathbb{R}^{q}$ denotes the control constraint set, which is defined as 
 
\begin{equation}
\label{eq:control-constraint}
\mathcal U \coloneqq \{\boldsymbol{u}\in \mathbb{R}^{q}:\boldsymbol{u}_{min}\le \boldsymbol{u} \le \boldsymbol{u}_{max} \}, \end{equation}
 with $\boldsymbol{u}_{min},\boldsymbol{u}_{max}\in \mathbb{R}^{q}$ (the vector inequalities are interpreted componentwise).
 
\begin{definition}[Class $\kappa$ function~\cite{Khalil:1173048}]
\label{def:class-k-f}
A continuous function $\alpha:[0,a)\to[0,+\infty],a>0$ is called a class $\kappa$ function if it is strictly increasing and $\alpha(0)=0.$
\end{definition}

\begin{definition}
\label{def:forward-inv}
A set $\mathcal C\subset \mathbb{R}^{n}$ is forward invariant for system \eqref{eq:affine-control-system} if its solutions for some $\boldsymbol{u} \in \mathcal U$ starting from any $\boldsymbol{x}(0) \in \mathcal C$ satisfy $\boldsymbol{x}(t) \in \mathcal C, \forall t \ge 0.$
\end{definition}

\begin{definition}
\label{def:relative-degree}
The relative degree of a differentiable function $b:\mathbb{R}^{n}\to\mathbb{R}$ is the minimum number of times we need to differentiate it along dynamics \eqref{eq:affine-control-system} until every component of $\boldsymbol{u}$ explicitly shows. 
\end{definition}

In this paper, safety is defined as the forward invariance of set $\mathcal C$. The relative degree of function $b$ is also referred to as the relative degree of constraint $b(\boldsymbol{x}) \ge 0$. For a constraint $b(\boldsymbol{x})\ge0$ with relative degree $m$, \ $b:\mathbb{R}^{n}\to\mathbb{R}$ and $\psi_{0}(\boldsymbol{x})\coloneqq b(\boldsymbol{x}),$ we can define a sequence of functions as $\psi_{i}:\mathbb{R}^{n}\to\mathbb{R},\ i\in \{1,...,m\}:$

\begin{equation}
\label{eq:sequence-f1}
\psi_{i}(\boldsymbol{x})\coloneqq\dot{\psi}_{i-1}(\boldsymbol{x})+\alpha_{i}(\psi_{i-1}(\boldsymbol{x})),\ i\in \{1,...,m\}, 
\end{equation}
where $\alpha_{i}(\cdot ),\ i\in \{1,...,m\}$ denotes a $(m-i)^{th}$ order differentiable class $\kappa$ function. A sequence of sets $\mathcal C_{i}$ are defined based on \eqref{eq:sequence-f1} as
\begin{equation}
\label{eq:sequence-set1}
\mathcal C_{i}\coloneqq \{\boldsymbol{x}\in\mathbb{R}^{n}:\psi_{i}(\boldsymbol{x})\ge 0\}, \ i\in \{0,...,m-1\}. 
\end{equation}

\begin{definition}[HOCBF~\cite{xiao2021high}]
\label{def:HOCBF}
Let $\psi_{i}(\boldsymbol{x}),\ i\in \{1,...,m\}$ be defined by \eqref{eq:sequence-f1} and $\mathcal C_{i},\ i\in \{0,...,m-1\}$ be defined by \eqref{eq:sequence-set1}. A function $b:\mathbb{R}^{n}\to\mathbb{R}$ is a High Order Control Barrier Function (HOCBF) with relative degree $m$ for system \eqref{eq:affine-control-system} if there exist $(m-i)^{th}$ order differentiable class $\kappa$ functions $\alpha_{i},\ i\in \{1,...,m\}$ such that
\begin{equation}
\label{eq:highest-HOCBF}
\begin{split}
\sup_{\boldsymbol{u}\in \mathcal U}[L_{f}^{m}b(\boldsymbol{x})+L_{g}L_{f}^{m-1}b(\boldsymbol{x})\boldsymbol{u}+O(b(\boldsymbol{x}))
+\\
\alpha_{m}(\psi_{m-1}(\boldsymbol{x}))]\ge 0,
\end{split}
\end{equation}
$\forall \boldsymbol{x}\in \mathcal C_{0}\cap,...,\cap \mathcal C_{m-1},$ where $O(\cdot)=\sum_{i=1}^{m-1}L_{f}^{i}(\alpha_{m-1}\circ\psi_{m-i-1})(\boldsymbol{x})$; $L_{f}^{m}$ denotes the $m$-th Lie derivative along $f$ and $L_{g}$ denotes the matrix of Lie derivatives along the columns of $g$. $\psi_{i}(\boldsymbol{x})\ge0$ is referred to as the $i^{th}$ order HOCBF constraint. We assume that $L_{g}L_{f}^{m-1}b(\boldsymbol{x})\boldsymbol{u}\ne0$ on the boundary of set $\mathcal C_{0}\cap,...,\cap \mathcal C_{m-1}.$ 
\end{definition}

\begin{theorem}[Safety Guarantee~\cite{xiao2021high}]
\label{thm:safety-guarantee}
Given a HOCBF $b(\boldsymbol{x})$ from Def. \ref{def:HOCBF} with corresponding sets $\mathcal{C}_{0}, \dots,\mathcal {C}_{m-1}$ defined by \eqref{eq:sequence-set1}, if $\boldsymbol{x}(0) \in \mathcal {C}_{0}\cap \dots \cap \mathcal {C}_{m-1},$ then any Lipschitz controller $\boldsymbol{u}$ that satisfies the constraint in \eqref{eq:highest-HOCBF}, $\forall t\ge 0$ renders $\mathcal {C}_{0}\cap \dots \cap \mathcal {C}_{m-1}$ forward invariant for system \eqref{eq:affine-control-system}, $i.e., \boldsymbol{x} \in \mathcal {C}_{0}\cap \dots \cap \mathcal {C}_{m-1}, \forall t\ge 0.$
\end{theorem}

\begin{definition}[CLF~\cite{ames2012control}]
\label{def:control-l-f}
A continuously differentiable function $V:\mathbb{R}^{n}\to\mathbb{R}$ is an exponentially stabilizing Control Lyapunov Function (CLF) for system \eqref{eq:affine-control-system} if there exist constants $c_{1}>0, c_{2}>0,c_{3}>0$ such that for $\forall \boldsymbol{x} \in \mathbb{R}^{n}, c_{1}\left \|  \boldsymbol{x} \right \| ^{2} \le V(\boldsymbol{x}) \le c_{2}\left \|  \boldsymbol{x} \right \| ^{2},$
\begin{equation}
\label{eq:clf}
\inf_{\boldsymbol{u}\in \mathcal U}[L_{f}V(\boldsymbol{x})+L_{g}V(\boldsymbol{x})\boldsymbol{u}+c_{3}V(\boldsymbol{x})]\le 0.
\end{equation}
\end{definition}

The existing works \cite{nguyen2016exponential},\cite{xiao2021high} combine HOCBFs \eqref{eq:highest-HOCBF} for systems with high relative degree with quadratic costs to form safety-critical optimization problems.
CLFs \eqref{eq:clf} can also be incorporated in optimization problems (see \cite{xiao2022sufficient},\cite{xiao2021adaptive}) if exponential convergence of some states is desired. In these works, time is discretized into time intervals,
and an optimization problem with constraints
given by HOCBFs and CLFs is solved in each
time interval. Since the state value is fixed
at the beginning of the interval, these
constraints are linear in control, therefore each optimization problem is a QP. The optimal control obtained by solving each QP is applied at the beginning of the interval and held constant for the
whole interval. During each interval, the state is updated using dynamics \eqref{eq:affine-control-system}. This method, which is referred to as CBF-CLF-QP, works conditioned on the fact that solving the QP at every time interval is feasible. However, this is not guaranteed, in particular
under tight or time-varying control bounds. 
 The authors of \cite{xiao2021adaptive} proposed a new type of HOCBF called PACBF, which introduced a time-varying penalty variable $p_{i}(t)$ in front of the class $\kappa$ function in the $i^{th}$ order HOCBF constraint ($i\in \{1,\dots,m\}$), trying to maximize the feasibility of solving CBF-CLF-QPs. However, the formulation of PACBFs requires the design of many additional constraints. Defining such constraints may not be straightforward, and may result in complicated parameter-tuning processes. To address this issue, we develop a new type of adaptive CBFs, called Auxiliary-Variable Adaptive Control Barrier Functions (AVCBFs), which is described in the next section.

\section{Auxiliary-Variable Adaptive Control Barrier Functions}
\label{sec:AVBCBF}

In this section, we introduce Auxiliary-Variable Adaptive Control Barrier Functions (AVCBFs) for safety-critical control.
We start with a simple example to motivate the need for AVCBFs.

\subsection{Motivation Example: Simplified Adaptive Cruise Control}
\label{subsec:SACC-problem}

Consider a Simplified Adaptive Cruise Control (SACC) problem with the dynamics of ego vehicle expressed as 
\begin{small}
\begin{equation}
\label{eq:SACC-dynamics}
\underbrace{\begin{bmatrix}
\dot{z}(t) \\
\dot{v}(t) 
\end{bmatrix}}_{\dot{\boldsymbol{x}}(t)}  
=\underbrace{\begin{bmatrix}
 v_{p}-v(t) \\
 0
\end{bmatrix}}_{f(\boldsymbol{x}(t))} 
+ \underbrace{\begin{bmatrix}
  0 \\
  1 
\end{bmatrix}}_{g(\boldsymbol{x}(t))}u(t),
\end{equation}
\end{small}
where $v_{p}>0, v(t)>0$ denote the velocity of the lead vehicle (constant velocity) and ego vehicle, respectively, $z(t)$ denotes the distance between the lead and ego vehicle and $u(t)$ denotes the acceleration (control) of ego vehicle, subject to the control constraints
\begin{equation}
\label{eq:simple-control-constraint}
u_{min}\le u(t) \le u_{max}, \forall t \ge0,
\end{equation}
where $u_{min}<0$ and $u_{max}>0$ are the minimum and maximum control input, respectively.

 For safety, we require that $z(t)$ always be greater than or equal to the safety distance denoted by $l_{p}>0,$ i.e., $z(t)\ge l_{p}, \forall t \ge 0.$ Based on Def. \ref{def:HOCBF}, let $\psi_{0}(\boldsymbol{x})\coloneqq b(\boldsymbol{x})=z(t)-l_{p}.$ From \eqref{eq:sequence-f1} and \eqref{eq:sequence-set1}, since the relative degree of $b(\boldsymbol{x})$ is 2, we have
\begin{equation}
\label{eq:SACC-HOCBF-sequence}
\begin{split}
&\psi_{1}(\boldsymbol{x})\coloneqq v_{p}-v(t)+k_{1}\psi_{0}(\boldsymbol{x})\ge 0
,\\
&\psi_{2}(\boldsymbol{x})\coloneqq -u(t)+k_{1}(v_{p}-v(t))+k_{2}\psi_{1}(\boldsymbol{x})\ge 0,
\end{split}
\end{equation}
where $\alpha_{1}(\psi_{0}(\boldsymbol{x}))\coloneqq k_{1}\psi_{0}(\boldsymbol{x}), \alpha_{2}(\psi_{1}(\boldsymbol{x}))\coloneqq k_{2}\psi_{1}(\boldsymbol{x}), k_{1}>0, k_{2}>0.$ The constant class $\kappa$ coefficients $k_{1},k_{2}$ are always chosen small to equip ego vehicle with a conservative control strategy to keep it safe, i.e., smaller $k_{1},k_{2}$ make ego vehicle brake earlier (see \cite{xiao2021high}). Suppose we wish to minimize the energy cost as $\int_{0}^{T} u^{2}(t)dt.$ We can then formulate the cost in the QP with constraint $\psi_{2}(\boldsymbol{x})\ge0$ and control input constraint \eqref{eq:simple-control-constraint} to get the optimal controller for the SACC problem. However, the feasible set of input can easily become empty if $u(t)\le k_{1}(v_{p}-v(t))+k_{2}\psi_{1}(\boldsymbol{x})<u_{min}$,  which causes infeasibility of the optimization. In \cite{xiao2021adaptive}, the authors introduced penalty variables in front of class $\kappa$ functions to enhance the feasibility. This approach defines $\psi_{0}(\boldsymbol{x})\coloneqq b(\boldsymbol{x})=z(t)-l_{p}$ as PACBF and other constraints can be further defined as
\begin{equation}
\label{eq:SACC-PACBF-sequence}
\begin{split}
\psi_{1}(\boldsymbol{x},p_{1}(t))&\coloneqq v_{p}-v(t)+p_{1}(t)k_{1}\psi_{0}(\boldsymbol{x})\ge 0,\\
\psi_{2}(\boldsymbol{x},p_{1}(t),&\boldsymbol{\nu})\coloneqq \nu_{1}(t)k_{1}\psi_{0}(\boldsymbol{x})+p_{1}(t)k_{1}(v_{p}\\
-v(t))&-u(t)+\nu_{2}(t)k_{2}\psi_{1}(\boldsymbol{x},p_{1}(t))\ge 0,
\end{split}
\end{equation}
where $\nu_{1}(t)=\dot{p}_{1}(t),\nu_{2}(t)=p_{2}(t), p_{1}(t)\ge0,p_{2}(t)\ge0,\boldsymbol{\nu}=(\nu_{1}(t),\nu_{2}(t)).$ $p_{1}(t),p_{2}(t)$ are time-varying penalty variables, which alleviate the conservativeness of the control strategy and $\nu_{1}(t),\nu_{2}(t)$ are auxiliary inputs, which relax the constraints for $u(t)$ in $\psi_{2}(\boldsymbol{x},p_{1}(t),\boldsymbol{\nu})\ge0$ and \eqref{eq:simple-control-constraint}. However, in practice, we need to define several additional constraints to make PACBF valid as shown in Eqs. (24)-(27) in \cite{xiao2021adaptive}. First, we need to define HOCBFs ($b_{1}(p_{1}(t))=p_{1}(t),b_{2}(p_{2}(t))=p_{2}(t))$ based on Def. \ref{def:HOCBF} to ensure $p_{1}(t)\ge0,p_{2}(t)\ge0.$ Next we need to define HOCBF ($b_{3}(p_{1}(t))=p_{1,max}-p_{1}(t)$) to confine the value of $p_{1}(t)$ in the range $[0,p_{1,max}].$ We also need to define CLF ($V(p_{1}(t))=(p_{1}(t)-p_{1}^{\ast})^{2}$) based on Def. \ref{def:control-l-f} to keep $p_{1}(t)$ close to a small value $p_{1}^{\ast}.$ $b_{3}(p_{1}(t)), V(p_{1}(t))$ are necessary since $\psi_{0}(\boldsymbol{x},p_{1}(t))\coloneqq p_{1}(t)k_{1}\psi_{0}(\boldsymbol{x})$ in first constraint in \eqref{eq:SACC-PACBF-sequence} is not a class $\kappa$ function with respect to $\psi_{0}(\boldsymbol{x}),$ i.e., $p_{1}(t)k_{1}\psi_{0}(\boldsymbol{x})$ is not guaranteed to strictly increase since $\psi_{0}(\boldsymbol{x},p_{1}(t))$ is in fact a class $\kappa$ function with respect to $p_{1}(t)\psi_{0}(\boldsymbol{x})$, which is against Thm. \ref{thm:safety-guarantee}, therefore $\psi_{1}(\boldsymbol{x},p_{1}(t))\ge 0$ in \eqref{eq:SACC-PACBF-sequence} may not guarantee $\psi_{0}(\boldsymbol{x})\ge 0.$ This illustrates why we have to limit the growth of $p_{1}(t)$ by defining $b_{3}(p_{1}(t)),V(p_{1}(t)).$ However, the way to choose appropriate values for $p_{1,max},p_{1}^{\ast}$ is not straightforward. We can imagine as the relative degree of $b(\boldsymbol{x})$ gets higher, the number of additional constraints we should define also gets larger, which results in complicated parameter-tuning process. To address this issue, we introduce $a_{1}(t),a_{2}(t)$ in the form
\begin{small}
\begin{equation}
\label{eq:SACC-AVBCBF-sequence}
\begin{split}
\psi_{1}(\boldsymbol{x},\boldsymbol{a},\dot{a}_{1}(t))\coloneqq a_{2}(t)(\dot{\psi}_{0}(\boldsymbol{x},a_{1}(t))
+k_{1}\psi_{0}(\boldsymbol{x},a_{1}(t)))\ge 0,\\
\psi_{2}(\boldsymbol{x},\boldsymbol{a},\dot{a}_{1}(t),\boldsymbol{\nu})\coloneqq \nu_{2}(t)\frac{\psi_{1}(\boldsymbol{x},\boldsymbol{a},\dot{a}_{1}(t))}{a_{2}(t)} +a_{2}(t)(\nu_{1}(t)(z(t)\\
-l_{p})+2\dot{a}_{1}(t)(v_{p}-v(t))-a_{1}(t)u(t)+k_{1}\dot{\psi}_{0}(\boldsymbol{x},a_{1}(t)))\\
+k_{2}\psi_{1}(\boldsymbol{x},\boldsymbol{a},\dot{a}_{1}(t))\ge 0, 
\end{split}
\end{equation}
\end{small}
where $\psi_{0}(\boldsymbol{x},a_{1}(t))\coloneqq a_{1}(t)b (\boldsymbol{x})=a_{1}(t)(z(t)-l_{p}),\boldsymbol{\nu}=[\nu_{1}(t),\nu_{2}(t)]^{T}=[\ddot{a}_{1}(t),\dot{a}_{2}(t)]^{T},\boldsymbol{a}=[a_{1}(t),a_{2}(t)]^{T},$ $a_{1}(t),a_{2}(t)$ are time-varying auxiliary variables. Since $\psi_{0}(\boldsymbol{x},a_{1}(t))\ge0,\psi_{1}(\boldsymbol{x},\boldsymbol{a},\dot{a}_{1}(t))\ge 0$ will not be against $b(\boldsymbol{x})\ge 0,\dot{\psi}_{0}(\boldsymbol{x},a_{1}(t))
+k_{1}\psi_{0}(\boldsymbol{x},a_{1}(t))\ge 0$ iff $a_{1}(t)>0,a_{2}(t)>0,$ we need to define HOCBFs for auxiliary variables to make $a_{1}(t)>0,a_{2}(t)>0,$ which will be illustrated in Sec. \ref{sec:AVCBFs}.  $\nu_{1}(t),\nu_{2}(t)$ are auxiliary inputs which are used to alleviate the restriction of constraints for $u(t)$ in $\psi_{2}(\boldsymbol{x},\boldsymbol{a},\dot{a}_{1}(t),\boldsymbol{\nu})\ge0$ and \eqref{eq:simple-control-constraint}. Different from the first constraint in \eqref{eq:SACC-PACBF-sequence}, $k_{1}\psi_{0}(\boldsymbol{x},a_{1}(t))$ is still a class $\kappa$ function with respect to $\psi_{0}(\boldsymbol{x},a_{1}(t)),$ therefore we do not need to define additional HOCBFs and CLFs like $b_{3}(p_{1}(t)),V(p_{1}(t))$ to limit the growth of $a_{1}(t).$
We can rewrite $\psi_{1} (\boldsymbol{x},\boldsymbol{a},\dot{a}_{1}(t))$ in \eqref{eq:SACC-AVBCBF-sequence} as
\begin{equation}
\label{eq:SACC-AVBCBF-sequence-rewrite}
\begin{split}
\psi_{1}(\boldsymbol{x},\boldsymbol{a},\dot{a}_{1}(t))\coloneqq a_{2}(t)a_{1}(t)(v_{p}-v(t)\\
+k_{1}(1+\frac{\dot{a}_{1}(t)}{k_{1}a_{1}(t)})b(\boldsymbol{x}))\ge 0.
\end{split}
\end{equation}
Compared to the first constraint in \eqref{eq:SACC-HOCBF-sequence}, $\frac{\dot{a}_{1}(t)}{a_{1}(t)}$ is a time-varying auxiliary term to alleviate the conservativeness of control that small $k_{1}$ originally has, which shows the adaptivity of auxiliary terms to constant class $\kappa$ coefficients. 

\subsection{Adaptive HOCBFs for Safety:\ AVCBFs}
\label{sec:AVCBFs}

Motivated by the SACC example in Sec. \ref{subsec:SACC-problem}, given a function $b:\mathbb{R}^{n}\to\mathbb{R}$ with relative degree $m$ for system \eqref{eq:affine-control-system} and a time-varying vector $\boldsymbol{a}(t)\coloneqq [a_{1}(t),\dots,a_{m}(t)]^{T}$ with positive components called auxiliary variables, the key idea in converting a regular HOCBF into an adaptive
one without defining excessive constraints is to place one auxiliary variable in front of each function in \eqref{eq:sequence-f1} similar to \eqref{eq:SACC-AVBCBF-sequence}. 
As described in Sec. \ref{subsec:SACC-problem}, we only need to define HOCBFs for auxiliary variables to ensure each $a_{i}(t)>0, i \in \{1,...,m\}.$ To realize this, we need to define auxiliary systems that contain auxiliary states $\boldsymbol{\pi}_{i}(t)$ and inputs $\nu_{i}(t)$, through which systems we can extend each HOCBF to desired relative degree, just like $b(\boldsymbol{x})$ has relative degree $m$
with respect to the dynamics \eqref{eq:affine-control-system}. Consider $m$ auxiliary systems in the form 
\begin{equation}
\label{eq:virtual-system}
\dot{\boldsymbol{\pi}}_{i}=F_{i}(\boldsymbol{\pi}_{i})+G_{i}(\boldsymbol{\pi}_{i})\nu_{i}, i \in \{1,...,m\},
\end{equation}
where $\boldsymbol{\pi}_{i}(t)\coloneqq [\pi_{i,1}(t),\dots,\pi_{i,m+1-i}(t)]^{T}\in \mathbb{R}^{m+1-i}$ denotes an auxiliary state with $\pi_{i,j}(t)\in \mathbb{R}, j \in \{1,...,m+1-i\}.$ $\nu_{i}\in \mathbb{R}$ denotes an auxiliary input for \eqref{eq:virtual-system}, $F_{i}:\mathbb{R}^{m+1-i}\to\mathbb{R}^{m+1-i}$ and $G_{i}:\mathbb{R}^{m+1-i}\to\mathbb{R}^{m+1-i}$ are locally Lipschitz. For simplicity, we just build up the connection between an auxiliary variable and the system as $a_{i}(t)=\pi_{i,1}(t), \dot{\pi}_{i,1}(t)=\pi_{i,2}(t),\dots,\dot{\pi}_{i,m-i}(t)=\pi_{i,m+1-i}(t)$ and make $\dot{\pi}_{i,m+1-i}(t)=\nu_{i},$ then we can define many specific HOCBFs $h_{i}$ to enable $a_{i}(t)$ to be positive. 

Given a function $h_{i}:\mathbb{R}^{m+1-i}\to\mathbb{R},$ we can define a sequence of functions $\varphi_{i,j}:\mathbb{R}^{m+1-i}\to\mathbb{R}, i \in\{1,...,m\}, j \in\{1,...,m+1-i\}:$
\begin{equation}
\label{eq:virtual-HOCBFs}
\varphi_{i,j}(\boldsymbol{\pi}_{i})\coloneqq\dot{\varphi}_{i,j-1}(\boldsymbol{\pi}_{i})+\alpha_{i,j}(\varphi_{i,j-1}(\boldsymbol{\pi}_{i})),
\end{equation}
where $\varphi_{i,0}(\boldsymbol{\pi}_{i})\coloneqq h_{i}(\boldsymbol{\pi}_{i}),$ $\alpha_{i,j}(\cdot)$ are $(m+1-i-j)^{th}$ order differentiable class $\kappa$ functions. Sets $\mathcal{B}_{i,j}$ are defined as
\begin{equation}
\label{eq:virtual-sets}
\mathcal B_{i,j}\coloneqq \{\boldsymbol{\pi}_{i}\in\mathbb{R}^{m+1-i}:\varphi_{i,j}(\boldsymbol{\pi}_{i})>0\}, \ j\in \{0,...,m-i\}. 
\end{equation}
Let $\varphi_{i,j}(\boldsymbol{\pi}_{i}),\ j\in \{1,...,m+1-i\}$ and $\mathcal B_{i,j},\ j\in \{0,...,m-i\}$ be defined by \eqref{eq:virtual-HOCBFs} and \eqref{eq:virtual-sets} respectively. By Def. \ref{def:HOCBF}, a function $h_{i}:\mathbb{R}^{m+1-i}\to\mathbb{R}$ is a HOCBF with relative degree $m+1-i$ for system \eqref{eq:virtual-system} if there exist class $\kappa$ functions $\alpha_{i,j},\ j\in \{1,...,m+1-i\}$ as in \eqref{eq:virtual-HOCBFs} such that
\begin{small}
\begin{equation}
\label{eq:highest-SHOCBF}
\begin{split}
\sup_{\nu_{i}\in \mathbb{R}}[L_{F_{i}}^{m+1-i}h_{i}(\boldsymbol{\pi}_{i})+L_{G_{i}}L_{F_{i}}^{m-i}h_{i}(\boldsymbol{\pi}_{i})\nu_{i}+O_{i}(h_{i}(\boldsymbol{\pi}_{i}))\\
+ \alpha_{i,m+1-i}(\varphi_{i,m-i}(\boldsymbol{\pi}_{i}))] \ge \epsilon,
\end{split}
\end{equation}
\end{small}
$\forall\boldsymbol{\pi}_{i}\in \mathcal B_{i,0}\cap,...,\cap \mathcal B_{i,m-i}$. $O_{i}(\cdot)=\sum_{j=1}^{m-i}L_{F_{i}}^{j}(\alpha_{i,m-i}\circ\varphi_{i,m-1-i})(\boldsymbol{\pi}_{i}) $ where $\circ$ denotes the composition of functions. $\epsilon$ is a positive constant which can be infinitely small. 

\begin{remark}
\label{rem:safety-guarantee-2}
If $h_{i}(\boldsymbol{\pi}_{i})$ is a HOCBF illustrated above and $\boldsymbol{\pi}_{i}(0) \in \mathcal {B}_{i,0}\cap \dots \cap \mathcal {B}_{i,m-i},$ then satisfying constraint in \eqref{eq:highest-SHOCBF} is equivalent to making $\varphi_{i,m+1-i}(\boldsymbol{\pi}_{i}(t))\ge \epsilon>0, \forall t\ge 0.$ Based on
\eqref{eq:virtual-HOCBFs}, since $\boldsymbol{\pi}_{i}(0) \in \mathcal {B}_{i,m-i}$ (i.e., $\varphi_{i,m-i}(\boldsymbol{\pi}_{i}(0))>0),$ then we have $\varphi_{i,m-i}(\boldsymbol{\pi}_{i}(t))>0$ (If there exists a $t_{1}\in (0,t_{2}]$, which makes $\varphi_{i,m-i}(\boldsymbol{\pi}_{i}(t_{1}))=0,$ then we have $\dot{\varphi}_{i,m-i}((\boldsymbol{\pi}_{i}(t_{1}))>0\Leftrightarrow \varphi_{i,m-i}(\boldsymbol{\pi}_{i}(t_{1}^{-}))\varphi_{i,m-i}(\boldsymbol{\pi}_{i}(t_{1}^{+}))<0,$ which is against the definition of $\alpha_{i,m+1-i}(\cdot),$ therefore $\forall t_{1}>0, \varphi_{i,m-i}(\boldsymbol{\pi}_{i}(t_{1}))>0,$ note that $t_{1}^{-},t_{1}^{+}$ denote the left and right limit). Based on \eqref{eq:virtual-HOCBFs}, since $\boldsymbol{\pi}_{i}(0) \in \mathcal {B}_{i,m-1-i},$ then similarly we have $\varphi_{i,m-1-i}(\boldsymbol{\pi}_{i}(t))>0,\forall t\ge 0.$ Repeatedly, we have $\varphi_{i,0}(\boldsymbol{\pi}_{i}(t))>0,\forall t\ge 0,$ therefore the sets $\mathcal {B}_{i,0},\dots,\mathcal {B}_{i,m-i}$ are forward invariant.
\end{remark}

For simplicity, we can make $h_{i}(\boldsymbol{\pi}_{i})=\pi_{i,1}(t)=a_{i}(t).$ Based on Rem. \ref{rem:safety-guarantee-2}, each $a_{i}(t)$ will be positive.

The remaining question is how to define an adaptive HOCBF to guarantee $b(\boldsymbol{x})\ge0$ with the assistance of auxiliary variables. Let $\boldsymbol{\Pi}(t)\coloneqq [\boldsymbol{\pi}_{1}(t),\dots,\boldsymbol{\pi}_{m}(t)]^{T}$ and $\boldsymbol{\nu}\coloneqq [\nu_{1},\dots,\nu_{m}]^{T}$ denote the auxiliary states and control inputs of system \eqref{eq:virtual-system}. We can define a sequence of functions 
\begin{small}
\begin{equation}
\label{eq:AVBCBF-sequence}
\begin{split}
&\psi_{0}(\boldsymbol{x},\boldsymbol{\Pi}(t))\coloneqq a_{1}(t)b(\boldsymbol{x}),\\
&\psi_{i}(\boldsymbol{x},\boldsymbol{\Pi}(t))\coloneqq a_{i+1}(t)(\dot{\psi}_{i-1}(\boldsymbol{x},\boldsymbol{\Pi}(t))+\alpha_{i}(\psi_{i-1}(\boldsymbol{x},\boldsymbol{\Pi}(t)))),
\end{split}
\end{equation}
\end{small}
where $i \in \{1,...,m-1\}, \psi_{m}(\boldsymbol{x},\boldsymbol{\Pi}(t))\coloneqq \dot{\psi}_{m-1}(\boldsymbol{x},\boldsymbol{\Pi}(t))+\alpha_{m}(\psi_{m-1}(\boldsymbol{x},\boldsymbol{\Pi}(t))).$ We further define a sequence of sets $\mathcal{C}_{i}$ associated with \eqref{eq:AVBCBF-sequence} in the form 
\begin{equation}
\label{eq:AVBCBF-set}
\begin{split}
\mathcal C_{i}\coloneqq \{(\boldsymbol{x},\boldsymbol{\Pi}(t)) \in \mathbb{R}^{n} \times \mathbb{R}^{m}:\psi_{i}(\boldsymbol{x},\boldsymbol{\Pi}(t))\ge 0\}, 
\end{split}
\end{equation}
where $i \in \{0,...,m-1\}.$
Since $a_{i}(t)$ is a HOCBF with relative degree $m+1-i$ for \eqref{eq:virtual-system}, based on \eqref{eq:highest-SHOCBF}, we define a constraint set $\mathcal{U}_{\boldsymbol{a}}$ for $\boldsymbol{\nu}$ as 
\begin{small}
\begin{equation}
\label{eq:constraint-up}
\begin{split}
\mathcal{U}_{\boldsymbol{a}}(\boldsymbol{\Pi})\coloneqq \{\boldsymbol{\nu}\in\mathbb{R}^{m}:   L_{F_{i}}^{m+1-i}a_{i}+[L_{G_{i}}L_{F_{i}}^{m-i}a_{i}]\nu_{i}\\
+O_{i}(a_{i})+ \alpha_{i,m+1-i}(\varphi_{i,m-i}(a_{i})) \ge \epsilon, i\in \{1,\dots,m\}\},
\end{split}
\end{equation}
\end{small}
where $\varphi_{i,m-i}(\cdot)$ is defined similar to \eqref{eq:virtual-HOCBFs} and $a_{i}(t)$ is ensured positive. $\epsilon$ is a positive constant which can be infinitely small. 

\begin{definition}[AVCBF]
\label{def:AVBCBF}
Let $\psi_{i}(\boldsymbol{x},\boldsymbol{\Pi}(t)),\ i\in \{1,...,m\}$ be defined by \eqref{eq:AVBCBF-sequence} and $\mathcal C_{i},\ i\in \{0,...,m-1\}$ be defined by \eqref{eq:AVBCBF-set}. A function $b(\boldsymbol{x}):\mathbb{R}^{n}\to\mathbb{R}$ is an Auxiliary-Variable Adaptive Control Barrier Function (AVCBF) with relative degree $m$ for system \eqref{eq:affine-control-system} if every $a_{i}(t),i\in \{1,...,m\}$ is a
HOCBF with relative degree $m+1-i$ for the auxiliary system
\eqref{eq:virtual-system}, and there exist $(m-j)^{th}$ order differentiable class $\kappa$ functions $\alpha_{j},j\in \{1,...,m-1\}$
and a class $\kappa$ functions $\alpha_{m}$ s.t.
\begin{small}
\begin{equation}
\label{eq:highest-AVBCBF}
\begin{split}
\sup_{\boldsymbol{u}\in \mathcal{U},\boldsymbol{\nu}\in \mathcal{U}_{\boldsymbol{a}}}[\sum_{j=2}^{m-1}[(\prod_{k=j+1}^{m}a_{k})\frac{\psi_{j-1}}{a_{j}}\nu_{j}] + \frac{\psi_{m-1}}{a_{m}}\nu_{m} \\ +(\prod_{i=2}^{m}a_{i})b(\boldsymbol{x})\nu_{1} +(\prod_{i=1}^{m}a_{i})(L_{f}^{m}b(\boldsymbol{x})+L_{g}L_{f}^{m-1}b(\boldsymbol{x})\boldsymbol{u})\\+R(b(\boldsymbol{x}),\boldsymbol{\Pi})
+ \alpha_{m}(\psi_{m-1})] \ge 0,
\end{split}
\end{equation}
\end{small}
$\forall (\boldsymbol{x},\boldsymbol{\Pi})\in \mathcal C_{0}\cap,...,\cap \mathcal C_{m-1}$ and each $a_{i}>0, i\in\{1,\dots,m\}.$ In \eqref{eq:highest-AVBCBF}, $R(b(\boldsymbol{x}),\boldsymbol{\Pi})$ denotes the remaining Lie derivative terms of $b(\boldsymbol{x})$ (or $\boldsymbol{\Pi}$) along $f$ (or $F_{i},i\in\{1,\dots,m\}$) with degree less than $m$ (or $m+1-i$), which is similar to the form of $O(\cdot )$ in \eqref{eq:highest-HOCBF}.
\end{definition}

\begin{theorem}
\label{thm:safety-guarantee-3}
Given an AVCBF $b(\boldsymbol{x})$ from Def. \ref{def:AVBCBF} with corresponding sets $\mathcal{C}_{0}, \dots,\mathcal {C}_{m-1}$ defined by \eqref{eq:AVBCBF-set}, if $(\boldsymbol{x}(0),\boldsymbol{\Pi}(0)) \in \mathcal {C}_{0}\cap \dots \cap \mathcal {C}_{m-1},$ then if there exists solution of Lipschitz controller $(\boldsymbol{u},\boldsymbol{\nu})$ that satisfies the constraint in \eqref{eq:highest-AVBCBF} and also ensures $(\boldsymbol{x},\boldsymbol{\Pi})\in \mathcal {C}_{m-1}$ for all $t\ge 0,$ then $\mathcal {C}_{0}\cap \dots \cap \mathcal {C}_{m-1}$ will be rendered forward invariant for system \eqref{eq:affine-control-system}, $i.e., (\boldsymbol{x},\boldsymbol{\Pi}) \in \mathcal {C}_{0}\cap \dots \cap \mathcal {C}_{m-1}, \forall t\ge 0.$ Moreover, $b(\boldsymbol{x})\ge 0$ is ensured for all $t\ge 0.$
\end{theorem}

\begin{proof}
If $b(\boldsymbol{x})$ is an AVCBF that is $m^{th}$ order differentiable, then satisfying constraint in \eqref{eq:highest-AVBCBF} while ensuring $(\boldsymbol{x},\boldsymbol{\Pi})\in \mathcal {C}_{m-1}$ for all $t\ge 0$ is equivalent to make $\psi_{m-1}(\boldsymbol{x},\boldsymbol{\Pi})\ge 0, \forall t\ge 0.$ Since $a_{m}(t)>0$, we have $\frac{\psi_{m-1}(\boldsymbol{x},\boldsymbol{\Pi})}{a_{m}(t)}\ge 0.$ Based on
\eqref{eq:AVBCBF-sequence}, since $(\boldsymbol{x}(0),\boldsymbol{\Pi}(0)) \in \mathcal {C}_{m-2}$ (i.e., $\frac{\psi_{m-2}(\boldsymbol{x}(0),\boldsymbol{\Pi}(0))}{a_{m-1}(0)}\ge 0),a_{m-1}(t)>0,$ then we have $\psi_{m-2}(\boldsymbol{x},\boldsymbol{\Pi})\ge 0$ (The proof of this is similar to the proof in Rem. \ref{rem:safety-guarantee-2}), and also $\frac{\psi_{m-2}(\boldsymbol{x},\boldsymbol{\Pi})}{a_{m-1}(t)}\ge 0.$ Based on \eqref{eq:AVBCBF-sequence}, since $(\boldsymbol{x}(0),\boldsymbol{\Pi}(0)) \in \mathcal {C}_{m-3},a_{m-2}(t)>0$ then similarly we have $\psi_{m-3}(\boldsymbol{x},\boldsymbol{\Pi})\ge 0$ and $\frac{\psi_{m-3}(\boldsymbol{x},\boldsymbol{\Pi})}{a_{m-2}(t)}\ge 0,\forall t\ge 0.$ Repeatedly, we have $\psi_{0}(\boldsymbol{x},\boldsymbol{\Pi})\ge 0$ and $\frac{\psi_{0}(\boldsymbol{x},\boldsymbol{\Pi})}{a_{1}(t)}\ge 0,\forall t\ge 0.$ Therefore the sets $\mathcal {C}_{0},\dots,\mathcal {C}_{m-1}$ are forward invariant and $b(\boldsymbol{x})=\frac{\psi_{0}(\boldsymbol{x},\boldsymbol{\Pi})}{a_{1}(t)}\ge 0$ is ensured for all $t\ge 0$.
\end{proof}
Based on Thm. \ref{thm:safety-guarantee-3}, the safety regarding $b(\boldsymbol{x})=\frac{\psi_{0}(\boldsymbol{x},\boldsymbol{\Pi})}{a_{1}(t)}\ge 0$ is guaranteed.

\begin{remark}[Limitation of Approaches with Auxiliary Inputs]
\label{rem: PACBF-AVBCBF} 
Ensuring the satisfaction of the $i^{th}$ order AVCBF constraint as shown in \eqref{eq:AVBCBF-set} when $i\in\{1,\dots,m-1\},$ i.e., $\psi_{i}(\boldsymbol{x},\boldsymbol{\Pi})\ge 0$ will guarantee $\psi_{i-1}(\boldsymbol{x},\boldsymbol{\Pi})\ge 0$ based on the proof of Thm. \ref{thm:safety-guarantee-3}, which theoretically outperforms PACBF. However, both approaches can not ensure satisfying $\psi_{m}(\boldsymbol{x},\boldsymbol{\Pi})\ge 0$ will guarantee $\psi_{m-1}(\boldsymbol{x},\boldsymbol{\Pi})\ge 0$ since the growth of $\boldsymbol{\nu}_{i}$ is unbounded. Therefore in Thm. \ref{thm:safety-guarantee-3}, $(\boldsymbol{x},\boldsymbol{\Pi})\in \mathcal {C}_{m-1}$ for all $t\ge 0$ also needs to be satisfied to guarantee the forward invariance of the intersection of sets. 
\end{remark}

\subsection{Optimal Control with AVCBFs}
\label{subsec: optimal-control}
Consider an optimal control problem as
\begin{small}
\begin{equation}
\label{eq:cost-function-1}
\begin{split}
 \min_{\boldsymbol{u}} \int_{0}^{T} 
 D(\left \| \boldsymbol{u} \right \| )dt,
\end{split}
\end{equation}
\end{small}
where $\left \| \cdot \right \|$ denotes the 2-norm of a vector, $D(\cdot)$ is a strictly increasing function of its argument and $T>0$ denotes the ending time. Since we need to introduce auxiliary inputs $v_{i}$ to enhance the feasibility of optimization, we should reformulate the cost in \eqref{eq:cost-function-1} as
\begin{small}
\begin{equation}
\label{eq:cost-function-2}
\begin{split}
 \min_{\boldsymbol{u},\boldsymbol{\nu}} \int_{0}^{T} 
 [D(\left \| \boldsymbol{u} \right \| )+\sum_{i=1}^{m}W_{i}(\nu_{i}-a_{i,w})^{2}]dt.
\end{split}
\end{equation}
\end{small}
In \eqref{eq:cost-function-2}, $W_{i}$ is a positive scalar and $a_{i,w}\in \mathbb{R}$ is the scalar to which we hope each auxiliary input $\nu_{i}$ converges. Both are chosen to tune the performance of the controller. We can formulate the CLFs, HOCBFs and AVCBFs introduced in Def. \ref{def:control-l-f}, Sec. \ref{sec:AVCBFs} and Def. \ref{def:AVBCBF} as constraints of the QP with cost function \eqref{eq:cost-function-2} to realize safety-critical control. Next we will show AVCBFs can be used to enhance the feasibility of solving QP compared with classical HOCBFs in Def. \ref{def:HOCBF}.

In auxiliary system \eqref{eq:virtual-system}, if we define $a_{i}(t)=\pi_{i,1}(t)=1, \dot{\pi}_{i,1}(t)=\dot{\pi}_{i,2}(t)=\cdots=\dot{\pi}_{i,m+1-i}(t)=0,$ then the way we construct functions and sets in \eqref{eq:virtual-HOCBFs} and \eqref{eq:virtual-sets} are exactly the same as \eqref{eq:sequence-f1} and \eqref{eq:sequence-set1}, which means classical HOCBF is in fact one specific case of AVCBF. Assume that the highest order HOCBF constraint \eqref{eq:highest-HOCBF} conflicts with control input constraints \eqref{eq:control-constraint} at $t=t_{b},$ i.e., we can not find a feasible controller $u(t_{b})$ to satisfy \eqref{eq:highest-HOCBF} and \eqref{eq:control-constraint}. Instead, starting from a time slot $t=t_{a}$ which is just before $t=t_{b}$ ($t_{b}-t_{a}=\varepsilon$ where $\varepsilon$ is an infinitely small positive value), we exchange the control framework of classical HOCBF into AVCBF instantly. Suppose we can find appropriate hyperparameters to ensure two constraints in \eqref{eq:constraint-up} and \eqref{eq:highest-AVBCBF}
% \begin{small}
% \begin{equation}
% \label{eq:constraint-fea-12}
% \begin{split}
%  \nu_{i}
%   > \frac{-L_{F_{i}}^{m+1-i}a_{i}-O_{i}(a_{i})-\alpha_{i,m+1-i}(\varphi_{i,m-i}(a_{i}))}{L_{G_{i}}L_{F_{i}}^{m-i}a_{i}},\\
%   \sum_{j=2}^{m-1}[(\prod_{k=j+1}^{m}a_{k})\frac{\psi_{j-1}}{a_{j}}\nu_{j}] + \frac{\psi_{m-1}}{a_{m}}\nu_{m} +(\prod_{i=2}^{m}a_{i})b(\boldsymbol{x})\nu_{1} \\ \ge -(\prod_{i=1}^{m}a_{i})(L_{f}^{m}b(\boldsymbol{x})+L_{g}L_{f}^{m-1}b(\boldsymbol{x})\boldsymbol{u})-R(b(\boldsymbol{x}),\boldsymbol{\Pi}) \\
% - \alpha_{m}(\psi_{m-1}),  i\in \{1,\dots,m\}
% \end{split}
% \end{equation}
% \end{small}
are satisfied given $\boldsymbol{u}$ constrained by \eqref{eq:control-constraint} at $t_{b},$ then there exists solution $\boldsymbol{u}(t_{b})$ for the optimal control problem and the feasibility of solving QP is enhanced. Relying on AVCBF, We can discretize the whole time period $[0,T]$ into several small time intervals like $[t_{a},t_{b}]$ to maximize the feasibility of solving QP under safety constraints, which calls for the development of automatic parameter-tuning techniques in future.

Besides safety and feasibility, another benefit of using AVCBFs is that the conservativeness of the control strategy can also be ameliorated. For example, from \eqref{eq:AVBCBF-sequence}, we can rewrite $\psi_{i}(\boldsymbol{x},\boldsymbol{\Pi})\ge 0$ as
\begin{equation}
\label{eq:AVCBF-rewrite}
\begin{split}
\dot{\phi}_{i-1}(\boldsymbol{x},\boldsymbol{\Pi})+k_{i}(1+\frac{\dot{a}_{i}(t)}{k_{i}a_{i}(t)}) \phi_{i-1}(\boldsymbol{x},\boldsymbol{\Pi})\ge0,
\end{split}
\end{equation}
where $\phi_{i-1}(\boldsymbol{x},\boldsymbol{\Pi})=\frac{\psi_{i-1}(\boldsymbol{x},\boldsymbol{\Pi})}{a_{i}(t)},\alpha_{i}(\psi_{i-1}(\boldsymbol{x},\boldsymbol{\Pi}))=k_{i}a_{i}(t)\phi_{i-1}(\boldsymbol{x},\boldsymbol{\Pi}), k_{i}>0, i\in \{1,\dots,m\}.$ Similar to PACBFs, we require $1+\frac{\dot{a}_{i}(t)}{k_{i}a_{i}(t)}\ge0,$ which gives us $\dot{a}_{i}(t)+k_{i}a_{i}(t)\ge0.$
The term $\frac{\dot{a}_{i}(t)}{a_{i}(t)}$ can be adjusted adaptable  to ameliorate the conservativeness of control strategy that $k_{i}\phi_{i-1}(\boldsymbol{x},\boldsymbol{\Pi})$ may have, i.e., the ego vehicle can brake earlier or later given time-varying control constraint $\boldsymbol{u}_{min}(t)\le \boldsymbol{u} \le\boldsymbol{u}_{max}(t),$ which confirms the adaptivity of AVCBFs to control constraint and conservativeness of control strategy. 

\begin{remark}[Parameter-Tuning for AVCBFs]
\label{rem: parameter-tuning}
Based on the analysis of \eqref{eq:AVCBF-rewrite}, we require $\dot{a}_{i}(t)+k_{i}a_{i}(t)\ge0.$ If we define first order HOCBF constraint for $a_{i}(t)>0$ as $\dot{a}_{i}(t)+l_{i}a_{i}(t)\ge0,$ we should choose hyperparameter $l_{i}\le k_{i}$ to guarantee $\dot{a}_{i}(t)+k_{i}a_{i}(t)\ge\dot{a}_{i}(t)+l_{i}a_{i}(t)\ge 0.$ For simplicity, we can use $l_{i}=k_{i}.$ In cost function \eqref{eq:cost-function-2}, we can tune hyperparameters $W_{i}$ and $a_{i,w}$ to adjust the corresponding ratio $\frac{\dot{a}_{i}(t)}{a_{i}(t)}$ to change the performance of the optimal controller.
\end{remark}

\begin{remark}
\label{rem: sufficient-con}
Note that the satisfaction of the constraint in \eqref{eq:highest-AVBCBF} is a sufficient condition for the satisfaction of the original constraint $\psi_{0}(\boldsymbol{x},\boldsymbol{\Pi})>0,$ it is not necessary to introduce auxiliary variables as many as from $a_{1}(t)$ to $a_{m}(t),$ which allows us to choose an appropriate
number of auxiliary variables for the AVCBF constraints to reduce the complexity. In other words, the number of auxiliary variables can be less than or equal to the relative degree $m$.
\end{remark}
\section{ACC Problem Formulation}
\label{sec:ACC-Problem}

In this section, we consider the Adaptive Cruise Control (ACC) problem, which is more realistic than the SACC problem introduced in Sec. \ref{subsec:SACC-problem} and was investigated in case study in \cite{ames2016control}, \cite{xiao2021adaptive}.
\subsection{Vehicle Dynamics}
We consider a nonlinear vehicle dynamics in the form
\begin{small}
\begin{equation}
\label{eq:ACC-dynamics}
\underbrace{\begin{bmatrix}
\dot{z}(t) \\
\dot{v}(t) 
\end{bmatrix}}_{\dot{\boldsymbol{x}}(t)}  
=\underbrace{\begin{bmatrix}
 v_{p}-v(t) \\
 -\frac{1}{M}F_{r}(v(t))
\end{bmatrix}}_{f(\boldsymbol{x}(t))} 
+ \underbrace{\begin{bmatrix}
  0 \\
  \frac{1}{M} 
\end{bmatrix}}_{g(\boldsymbol{x}(t))}u(t)
\end{equation}
\end{small}
where $M$ denotes the mass of the ego vehicle and $v_{p}>0$ denotes the velocity of the lead vehicle. $z(t)$ denotes the distance between two vehicles and $F_{r}(v(t))=f_{0}sgn(v(t))+f_{1}v(t)+f_{2}v^{2}(t)$ denotes the resistance force as in \cite{Khalil:1173048}, where $f_{0},f_{1},f_{2}$ are positive scalars determined empirically and $v(t)>0$ denotes the velocity of the ego vehicle. The first term in $F_{r}(t)$ denotes the Coulomb friction force, the second term denotes the viscous friction force and the last term denotes the aerodynamic drag.
\subsection{Vehicle Limitations}
Vehicle limitations include vehicle constraints on safe distance, speed and acceleration.

\textbf{Safe distance constraint:} The distance is considered safe if $z(t)\ge l_{p}$ is satisfied $\forall t \in [0,T]$, where $l_{p}$ denotes the minimum distance two vehicles should maintain.

\textbf{Speed constraint:} The ego vehicle should achieve a desired speed  $v_{d}>0.$

\textbf{Acceleration constraint:} The ego vehicle should minimize the following cost
\begin{small}
\begin{equation}
\label{eq:minimal-u}
\min_{u(t)} \int_{0}^{T}(\frac{u(t)-F_{r}(v(t))}{M})^{2}dt 
\end{equation}
\end{small}
when the acceleration is constrained in the form 
\begin{equation}
\label{eq:constraint-u}
-c_{d}(t)Mg\le u(t) \le c_{a}(t)Mg, \forall t \in [0,T], 
\end{equation}
where $g$ denotes the gravity constant, $c_{d}(t)>0$ and $c_{a}(t)>0$ are deceleration and acceleration coefficients respectively.
\begin{problem}
\label{prob:ACC-prob}
Determine the optimal controller for the ego vehicle governed by dynamics \eqref{eq:ACC-dynamics} while subject to the vehicle constraints on safe distance, speed and acceleration. 
\end{problem}

To satisfy the constraint on speed, we define a CLF $V(\boldsymbol{x}(t)) \coloneqq(v(t)-v_{d})^{2}$ with $c_{1}=c_{2}=1$ to stabilize $v(t)$ to $v_{d}$ and formulate the relaxed constraint in \eqref{eq:clf} as
\begin{equation}
\label{eq:ACC-clf}
L_{f}V(\boldsymbol{x}(t))+L_{g}V(\boldsymbol{x}(t))u(t)+c_{3}V(\boldsymbol{x}(t))\le \delta(t), 
\end{equation}
where $\delta(t)$ is a relaxation that makes \eqref{eq:ACC-clf} a soft constraint.

To satisfy the constraints on safety distance and acceleration, we will define a continuous function $b(\boldsymbol{x}(t))=z(t)-l_{p}$ as AVCBF or PACBF to guarantee $b(\boldsymbol{x}(t))\ge 0$ and constraint \eqref{eq:constraint-u}, then formulate all constraints into QP to get the optimal controller.
\section{Implementation and Results}
\label{sec:Implementation}
In this section, we show how our proposed AVCBFs will provide the adaptivity and outperform the PACBFs in solving the optimal ACC problem under conservative control constraints. We consider the Prob. \ref{prob:ACC-prob} with time-varying control bounds \eqref{eq:constraint-u} (due to smoothness of vehicle tires and road surfaces), and implement AVCBFs or PACBFs as safety constraints for solving Prob. \ref{prob:ACC-prob} in MATLAB. We utilized ode45 to integrate the dynamic system for every $0.1s$ time-interval and quadprog to solve QP. Both methods show varying degrees of adaptivity to complexity of road conditions.

The parameters are $v_{p}=13.89m/s, v_{d}=24m/s, M=1650kg, g=9.81m/s^{2},z(0)=100m, l_{p}=10m, f_{0}=0.1N, f_{1}=5Ns/m, f_{2}=0.25Ns^{2}/m, c_{a}(t)=0.4.$
\subsection{Implementation with AVCBFs}
Define $b(\boldsymbol{x}(t))=z(t)-l_{p},$ the relative degree of $b(\boldsymbol{x}(t))$ with respect to dynamics \eqref{eq:ACC-dynamics} is 2. For simplicity, based on Rem. \ref{rem: sufficient-con}, we just introduce one auxiliary variable as $\boldsymbol{a}(t)=a_{1}(t).$ 
Motivated by Sec. \ref{sec:AVCBFs}, we define the auxiliary dynamics as
\begin{small}
\begin{equation}
\label{eq:Auxiliary-dynamics1}
\underbrace{\begin{bmatrix}
\dot{a_{1}}(t) \\
\dot{\pi}_{1,2}(t) 
\end{bmatrix}}_{\dot{\boldsymbol{\pi}}_{1}(t)}  
=\underbrace{\begin{bmatrix}
 \pi_{1,2}(t) \\
 0
\end{bmatrix}}_{F_{1}(\boldsymbol{{\pi}}_{1}(t))} 
+ \underbrace{\begin{bmatrix}
  0 \\
  1 
\end{bmatrix}}_{G_{1}(\boldsymbol{{\pi}}_{1}(t))}\nu_{1}(t).
\end{equation}
\end{small}
The HOCBFs for $a_{1}(t)$ are defined as 
\begin{equation}
\label{eq:SHOCBF-sequence-ACC}
\begin{split}
&\varphi_{1,0}(\boldsymbol{{\pi}}_{1}(t))\coloneqq a_{1}(t),\\
&\varphi_{1,1}(\boldsymbol{{\pi}}_{1}(t))\coloneqq \dot{\varphi}_{1,0}(\boldsymbol{{\pi}}_{1}(t))+l_{1}\varphi_{1,0}(\boldsymbol{{\pi}}_{1}(t)),\\
&\varphi_{1,2}(\boldsymbol{{\pi}}_{1}(t))\coloneqq \dot{\varphi}_{1,1}(\boldsymbol{{\pi}}_{1}(t))+l_{2}\varphi_{1,1}(\boldsymbol{{\pi}}_{1}(t)),
\end{split}
\end{equation}
where $\alpha_{1,1}(\cdot),\alpha_{1,2}(\cdot)$ are defined as linear functions. The AVCBFs are then defined as
\begin{equation}
\label{eq:AVBCBF-sequence-ACC}
\begin{split}
&\psi_{0}(\boldsymbol{x},\boldsymbol{{\pi}}_{1}(t))\coloneqq a_{1}(t)b(\boldsymbol{x}),\\
&\psi_{1}(\boldsymbol{x},\boldsymbol{{\pi}}_{1}(t))\coloneqq \dot{\psi}_{0}(\boldsymbol{x},\boldsymbol{{\pi}}_{1}(t))+k_{1}\psi_{0}(\boldsymbol{x},\boldsymbol{{\pi}}_{1}(t)),\\
&\psi_{2}(\boldsymbol{x},\boldsymbol{{\pi}}_{1}(t))\coloneqq \dot{\psi}_{1}(\boldsymbol{x},\boldsymbol{{\pi}}_{1}(t))+k_{2}\psi_{1}(\boldsymbol{x},\boldsymbol{{\pi}}_{1}(t)),
\end{split}
\end{equation}
where $\alpha_{1}(\cdot),\alpha_{2}(\cdot)$ are set as linear functions. By formulating constraints from HOCBFs \eqref{eq:SHOCBF-sequence-ACC}, AVCBFs \eqref{eq:AVBCBF-sequence-ACC}, CLF \eqref{eq:ACC-clf} and acceleration \eqref{eq:constraint-u}, we can define cost function 
 for QP as
 \begin{small}
\begin{equation}
\label{eq:AVBCBF-cost}
\begin{split}
\min_{u(t),\nu_{1}(t),\delta(t)} \int_{0}^{T}[(\frac{u(t)-F_{r}(v(t))}{M})^{2}\\+W_{1}(\nu_{1}(t)-a_{1,w})^{2}+Q\delta(t)^{2}]dt.
\end{split}
\end{equation}
\end{small}
Other parameters are set as $v(0)=6m/s, a_{1}(0)=1, \pi_{1,2}(0)=1, c_{3}=2, W_{1}=Q=1000,\epsilon=10^{-10}.$
\begin{figure}[ht]
    \centering
    \includegraphics[scale=0.58]{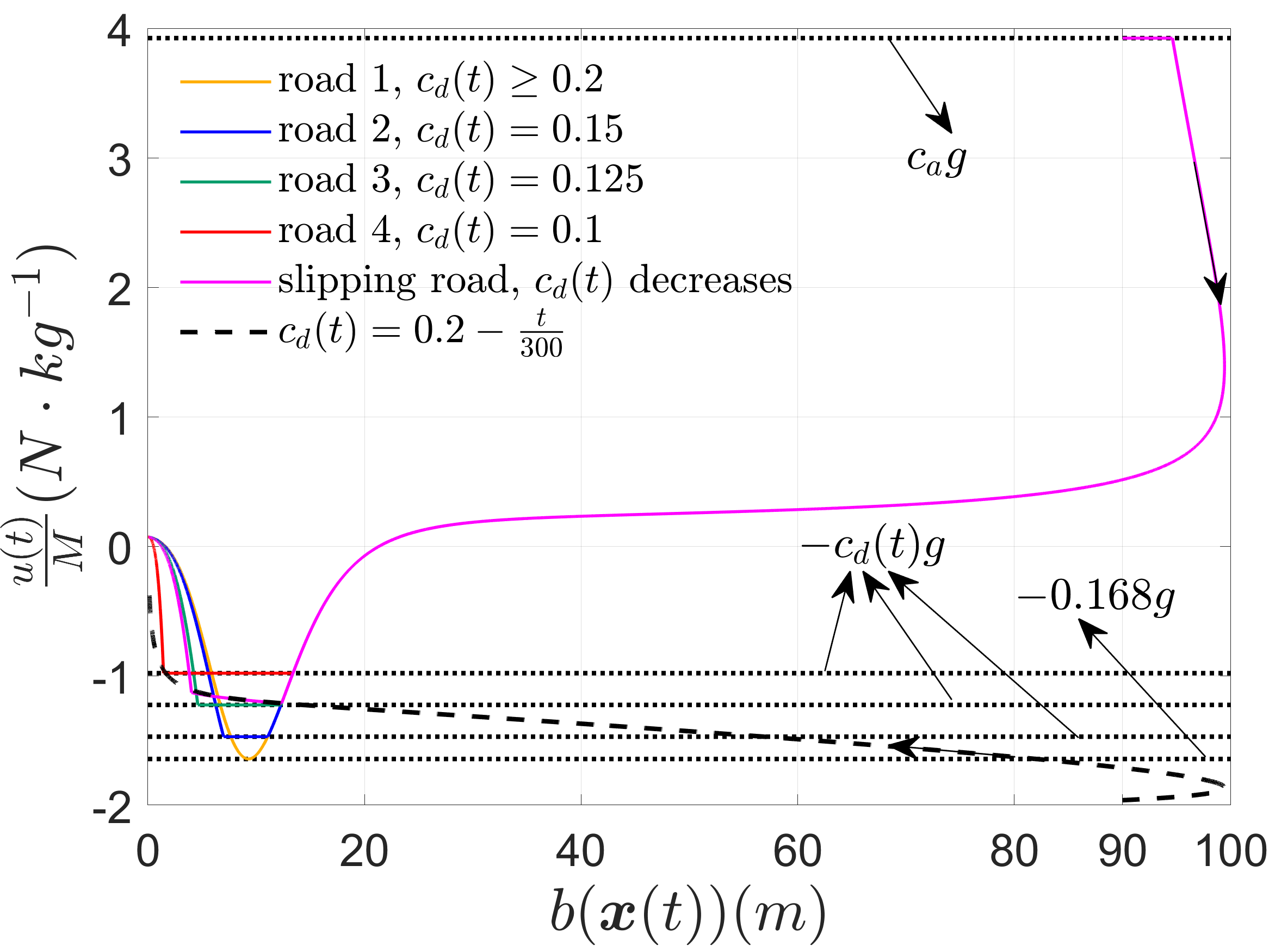}
    \caption{Control input $u(t)$ varies as $b(\boldsymbol{x(t)})$ goes to 0 under different lower control bounds. The arrows denote the changing trend for $b(\boldsymbol{x(t)})$ and $c_{d}(t)$ over time. $b(\boldsymbol{x(0)})=90$ and $b(\boldsymbol{x(t)})\ge 0$ implies safety. Hyperparameters are set as $k_{1}=k_{2}=l_{1}=l_{2}=0.1, a_{1,w}=1,T=50s.$ }
    \label{fig:AVBCBFs-braking}
\end{figure} 
We first test the adaptivity to deceleration by changing the lower control bound $-c_{d}(t)Mg.$ In each test in Fig. \ref{fig:AVBCBFs-braking}, we set deceleration coefficient $c_{d}(t)$ to different constant or linearly decreasing variable due to different road conditions. For each case, the ego vehicle first accelerates to the same velocity ($b(\boldsymbol{x(t)})$ increases to the same value), then starts to decelerate at the same time. Due to different braking capability, the ego vehicle reaches different maximum deceleration (denoted by arrows) and finally keeps a constant velocity the same as $v_{p}$. The safe distance $l_{p}$ is maintained for all $t\in[0,50s].$ Note that under poor road conditions (e.g., the road is very slippery or the smoothness of road varies), shown by red or magenta curves, QPs are still feasible by using AVCBFs method, which shows good adaptivity to control constraints. 
\begin{figure}[ht]
    \centering
    \includegraphics[scale=0.58]{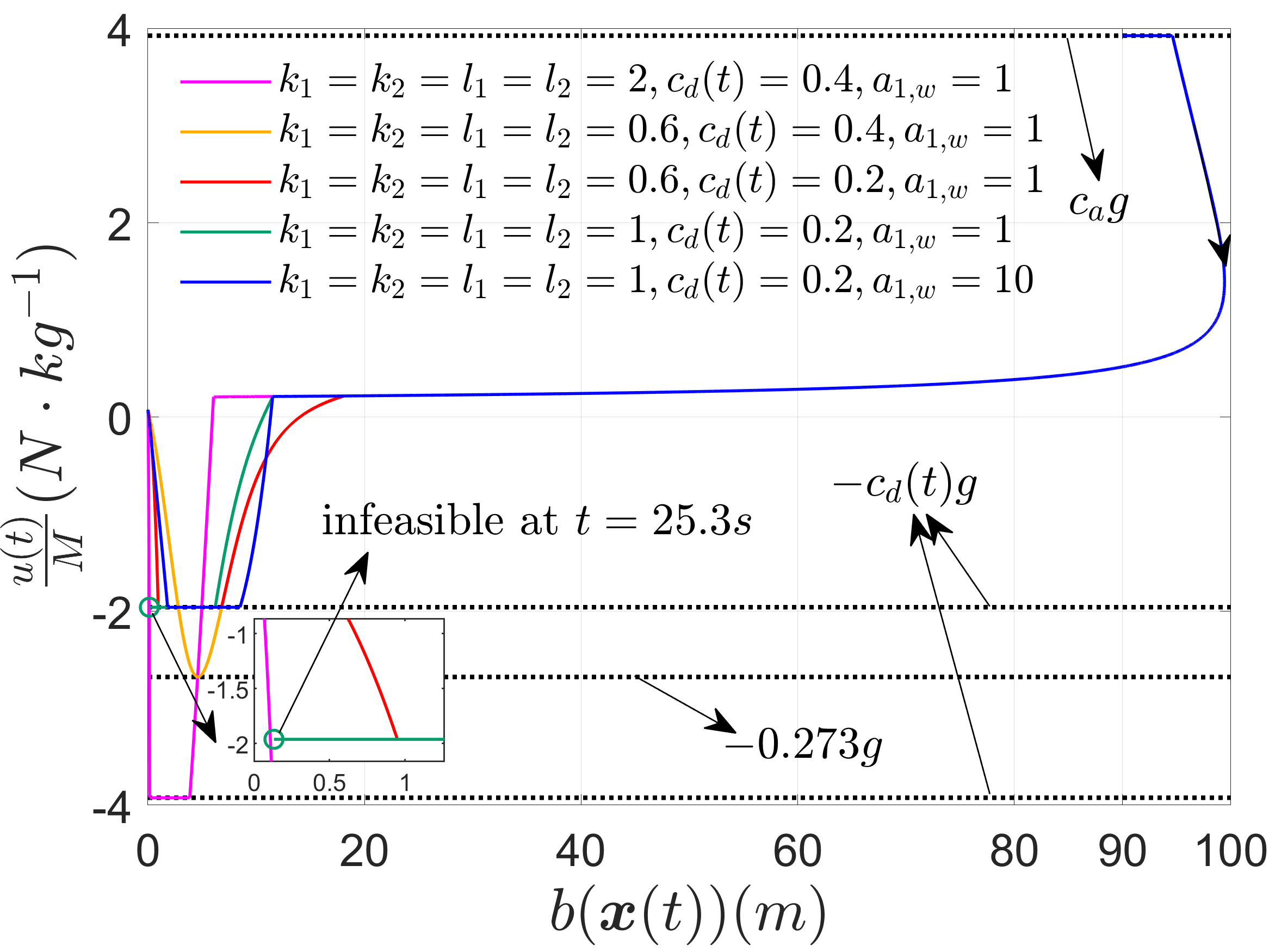}
    \caption{Control input $u(t)$ varies as $b(\boldsymbol{x(t)})$ goes to 0 under different lower control bounds. The arrows denote the changing trend for $b(\boldsymbol{x(t)})$ and $c_{d}(t)$ over 50 seconds. $b(\boldsymbol{x(0)})=90$ and $b(\boldsymbol{x(t)})\ge 0$ implies safety. Different sets of hyperparameters for class $\kappa$ functions are tested.}
    \label{fig:AVBCBFs-hyperparameters}
\end{figure} 

Next, we test the adaptivity to conservativeness of control strategy by changing the hyperparameters $k_{1},k_{2},l_{1},l_{2}$ inside the class $\kappa$ functions in \eqref{eq:SHOCBF-sequence-ACC},\eqref{eq:AVBCBF-sequence-ACC}. We also change $c_{d}(t)$ for different hyperparameters. For each case in Fig. \ref{fig:AVBCBFs-hyperparameters}, the ego car first accelerates to the same velocity ($b(\boldsymbol{x(t)})$ increases to the same value), then starts to decelerate at different time (the orange and red curves represent breaking earliest, while the green and blue curves represent breaking later and the magenta curve represents breaking latest). Due to different braking capability, the ego vehicle reaches different maximum deceleration (denoted by arrows) and finally keeps a constant velocity the same as $v_{p}$. The safe distance $l_{p}$ is maintained for all $t\in[0,50s].$ Note that from comparing magenta curve with orange curve, larger hyperparameters allow the ego vehicle to brake later and to reach larger deceleration, but if the $c_{d}(t)$ is very small (like green curve), larger hyperparameters always cause infeasibility of QPs (the ego vehicle does not have enough long distance to brake down safely, therefore infeasible at 25.3s). We can make ego vehicle brake faster (as blue curve shows) by adjusting $a_{1,w}$ in cost function \eqref{eq:AVBCBF-cost}, which shows good adaptivity of AVCBFs to managing control strategy.

\subsection{Implementation with PACBFs}

Similar to AVCBFs, we define $b(\boldsymbol{x}(t))=z(t)-l_{p}$ for PACBFs. We use the same penalty function and auxiliary dynamics in \cite{xiao2021adaptive} as $\boldsymbol{p}(t)=(p_{1}(t), p_{2}(t))$ and
\begin{equation}
\label{eq:AVBCBFs-PACBFs-1}
\begin{split}
\dot{p_{1}}(t)=\nu_{1}(t),\\
p_{2}(t)=\nu_{2}(t).
\end{split}
\end{equation}
To make $p_{1}(t)$ converge to a small enough value, we define CLF constraint as
\begin{equation}
\label{eq:HOCBFs-CLFs}
\begin{split}
2(p_{1}(t)-p_{1}^{\ast })\nu_{1}(t)+\rho (p_{1}(t)-p_{1}^{\ast})^{2}\le \delta_{p}(t),
\end{split}
\end{equation}
where $\rho=10, p_{1}^{\ast}=0.103, \delta_{p}(t)$ is relaxed variable. We define HOCBF constraints for $p_{1}(t)$ as 
\begin{equation}
\label{eq:Auxiliary-PACBFs}
\begin{split}
-\nu_{1}(t)+(3-p_{1}(t))\ge0,\\
\nu_{1}(t)+p_{1}(t)\ge0,
\end{split}
\end{equation}
which confines $p_{1}(t)$ into $[0,3].$ The PACBFs are then defined as 
\begin{equation}
\label{eq:AVBCBFs-PACBFs-2}
\begin{split}
&\psi_{0}(\boldsymbol{x})\coloneqq b(\boldsymbol{x}),\\
&\psi_{1}(\boldsymbol{x},\boldsymbol{p}(t))\coloneqq \dot{\psi}_{0}(\boldsymbol{x})+p_{1}(t)\psi_{0}(\boldsymbol{x})^{2},\\
&\psi_{2}(\boldsymbol{x},\boldsymbol{p}(t),\boldsymbol{\nu}(t))\coloneqq \dot{\psi}_{1}(\boldsymbol{x},\boldsymbol{p}(t))+\nu_{2}(t)\psi_{1}(\boldsymbol{x},\boldsymbol{p}(t))
\end{split}
\end{equation}
to guarantee the safety. By formulating constraints from HOCBFs \eqref{eq:Auxiliary-PACBFs}, CLFs \eqref{eq:HOCBFs-CLFs}\eqref{eq:ACC-clf}, PACBFs \eqref{eq:AVBCBFs-PACBFs-2} and acceleration \eqref{eq:constraint-u}, we can define cost function 
 for QP as
 \begin{small}
\begin{equation}
\label{eq:PACBF-cost}
\begin{split}
\min_{u(t),\nu_{1}(t),\nu_{2}(t),\delta(t),\delta_{p}(t)} \int_{0}^{T}[(\frac{u(t)-F_{r}(v(t))}{M})^{2}+W_{1}\nu_{1}(t)\\+W_{2}(\nu_{2}(t)-1)^{2}+Q\delta(t)^{2}+Q_{p}\delta_{p}(t)^{2}]dt. 
\end{split}
\end{equation}
\end{small}
Other parameters are set as $p_{1}(0)=0.103, p_{2}(0)=1, c_{3}=10, W_{1}=W_{2}=2e^{12},Q=Q_{p}=1.$
\subsection{Comparison between AVCBFs and PACBFs}
\begin{figure}[ht]
    \centering
    \includegraphics[scale=0.58]{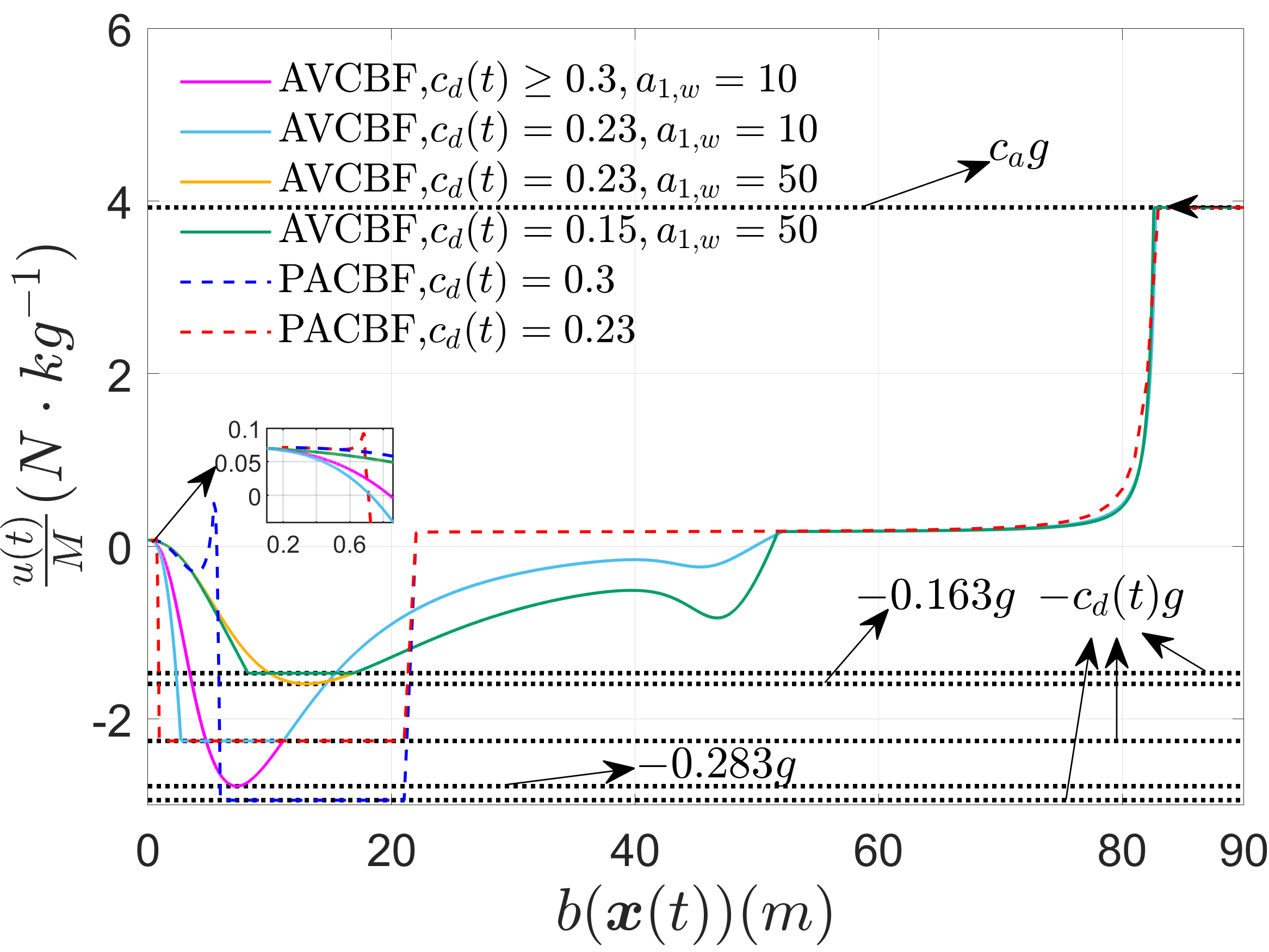}
    \caption{Control input $u(t)$ varies as $b(\boldsymbol{x(t)})$ goes to 0 under different lower control bounds. The arrows denote the changing trend for $b(\boldsymbol{x(t)})$ and $c_{d}(t)$ over 50 seconds. $b(\boldsymbol{x(0)})=90$ and $b(\boldsymbol{x(t)})\ge 0$ implies safety. Solid curves denote AVCBFs and dashed curves denote PACBFs. }
    \label{fig:AVBCBFs-PACBFs-1}
\end{figure} 
% \vspace{-0.5cm}

We compare our proposed AVCBFs with the state of the art (PACBFs) for a more urgent braking case by making initial velocity large as $v(0)=20m/s.$ In Fig. \ref{fig:AVBCBFs-PACBFs-1}, we change the lower control bound $-c_{d}(t)Mg$ to different value to compare both methods' adaptivity. Hyperparameter of CLF \eqref{eq:ACC-clf} for maganta and cyan curves is set as $c_{3}=70$ and for orange, green curves, we set $c_{3}=100$. Other hyperparameters for AVCBFs are set as $k_{1}=k_{2}=l_{1}=l_{2}=0.1, W_{1}=2e^{5}, Q=7e^{5}.$ We choose the cases denoted by cyan, orange and red curves where $c_{d}(t)=0.23$ in Fig. \ref{fig:AVBCBFs-PACBFs-1} and further analyze them in Fig. \ref{fig:AVBCBFs-PACBFs-2}. Based on Figs. \ref{fig:AVBCBFs-PACBFs-1} and \ref{fig:AVBCBFs-PACBFs-2}, by both methods, the ego vehicle first accelerates to desired velocity around $24m/s,$ then decelerates until $v=v_{p}.$ Even both methods work well for the cases shown in two figures, AVCBFs show more adaptivity to the road condition by simply adjusting $a_{1,w}$ thus the ego vehicle can brake earlier with even more strict limitation of deceleration (i.e., the smaller $c_{d}(t)$ due to more slippery road condition). For PACBFs, it is difficult to adjust hyperparameters to make ego vehicle adaptive to poor road conditions as $p_{1}(t)$ is always required to be a small enough by additional CLF \eqref{eq:HOCBFs-CLFs}. We also note that AVCBFs can generate a smoother optimal controller in Fig. \ref{fig:AVBCBFs-PACBFs-1}, as there is no noticeable overshoot (for blue and red curves by PACBFs, the controller shows overshooting strategy near the end, which is not smooth).
\begin{figure}[ht]
    \centering
    \includegraphics[scale=0.58]{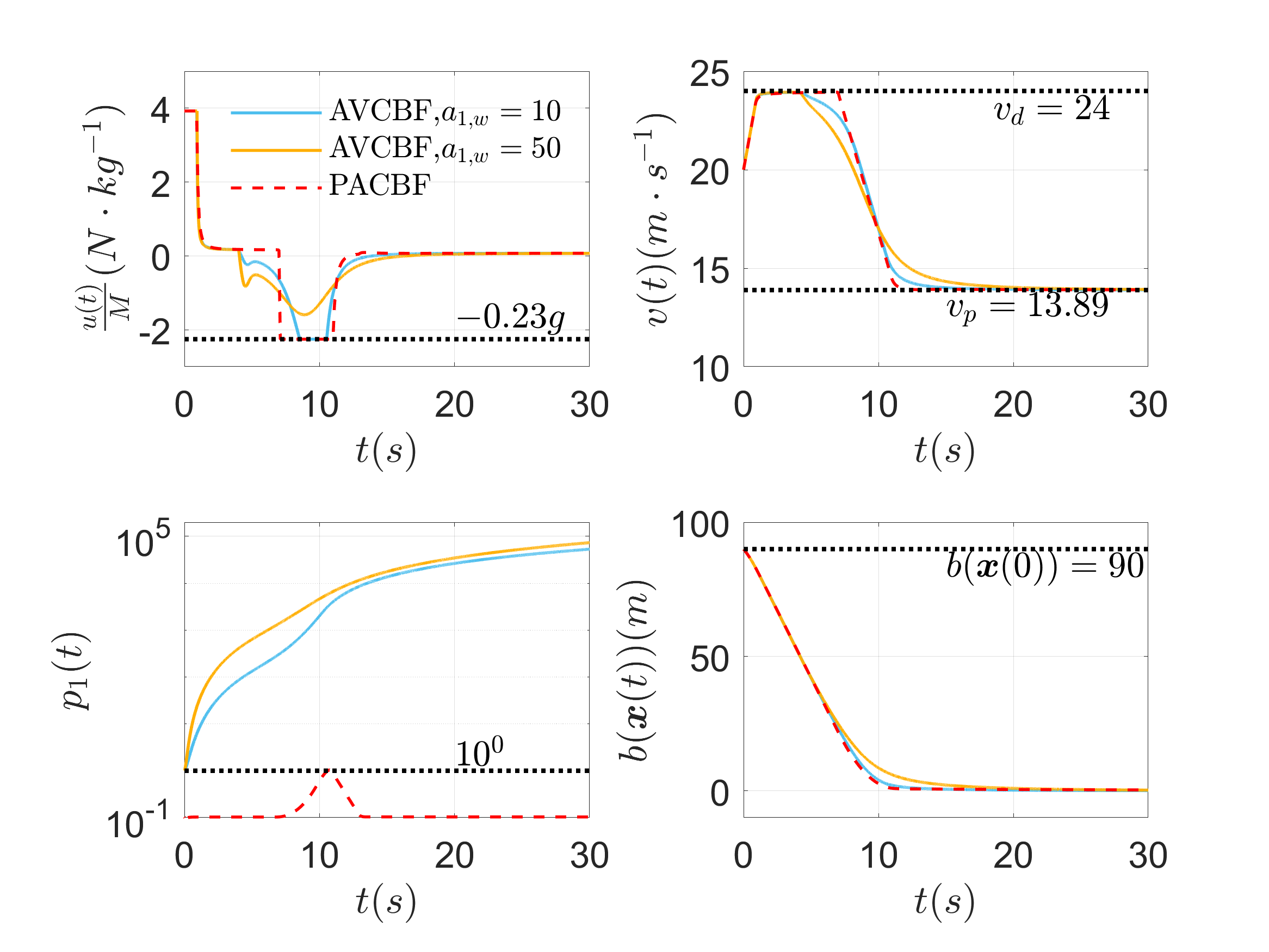}
    \caption{Control input $u(t)$, velocity $v(t)$, time-varying $p_{1}(t)$ and distance between two vehicles $b(\boldsymbol{x(t)})$ over 30 seconds for AVCBFs and PACBFs. $b(\boldsymbol{x(t)})\ge 0$ implies safety. Solid curves denote AVCBFs and dashed curve denotes PACBFs.}
    \label{fig:AVBCBFs-PACBFs-2}
\end{figure}

\section{Conclusion}
\label{sec:Conclusion}
We proposed Auxiliary-Variable Adaptive Control Barrier Functions (AVCBFs) for safety-critical optimal controller design, which addresses the notorious feasibility problem in the discretization-based QP approach in the case of time-varying and tight control bounds. Central to our approach is using auxiliary variables for the system dynamics and for the HOCBF constraints. We validated the proposed AVCBFs approach by applying it to a model of adaptive cruise control. Our proposed method generates a safe, smooth and adaptive controller by designing fewer constraints and simpler parameter tuning, which outperforms PACBFs as the baseline. One limitation of the proposed method is the fact that its hyperparameters in the cost function are not automatically tuned. Another limitation is that the feasibility of the optimization and system safety are not always guaranteed at the same time in the whole state space. We will address these limitations in future work by designing a fully automatic AVCBFs method. 
\bibliographystyle{IEEEtran}
\balance
\bibliography{references.bib}
\end{document}